\documentclass[11pt,a4paper,leqno]{amsart}
\usepackage{amsmath,amsthm,amscd,amssymb,amsfonts,amsbsy}
\usepackage{soul}
\usepackage{epigraph}
\usepackage[english]{babel}
\usepackage{latexsym}
\usepackage{exscale}
\usepackage{enumitem}
\usepackage{bbm}
\usepackage{enumitem}
\usepackage{graphicx}
\usepackage{cite}
\usepackage{txfonts}
\usepackage[colorlinks,citecolor=red,hypertexnames=false]{hyperref}
\usepackage{color}
\newcommand{\noop}[1]{}

\calclayout
\allowdisplaybreaks

\theoremstyle{plain}
\newtheorem{theorem}[equation]{Theorem}
\newtheorem{lemma}[equation]{Lemma}
\newtheorem{corollary}[equation]{Corollary}

\theoremstyle{definition}
\newtheorem{definition}[equation]{Definition}

\theoremstyle{remark}
\newtheorem{remarks}[equation]{Remarks}
\newtheorem{remark}[equation]{Remark}

\numberwithin{equation}{section}

\newcommand{\eps}{\varepsilon}

\newcommand{\re}{\mathbb{R}}
\newcommand{\R}{\mathbb{R}}
\newcommand{\rn}{{\mathbb{R}^n}}

\newcommand{\ree}{\mathbb{R}^{n+1}}
\newcommand{\om}{\Omega}
\newcommand{\hm}{\omega}

\DeclareMathOperator{\I}{I}
\DeclareMathOperator{\II}{II}

\DeclareMathOperator{\BMO}{BMO}

\def\Yint#1{\mathchoice
    {\YYint\displaystyle\textstyle{#1}}%
    {\YYint\textstyle\scriptstyle{#1}}%
    {\YYint\scriptstyle\scriptscriptstyle{#1}}%
    {\YYint\scriptscriptstyle\scriptscriptstyle{#1}}%
      \!\iint}
\def\YYint#1#2#3{{\setbox0=\hbox{$#1{#2#3}{\iint}$}
    \vcenter{\hbox{$#2#3$}}\kern-.51\wd0}}
\def\longdash{{-}\mkern-3.5mu{-}\mkern-3.5mu{-}} 
\def\tiltlongdash{\rotatebox[origin=c]{15}{$\longdash$}}
\def\fiint{\Yint\tiltlongdash}

\def\div{\mathop{\operatorname{div}}\nolimits}

\providecommand{\abs}[1]{ \left \lvert  #1 \right \rvert }
\providecommand{\no}[1]{  \lVert  #1  \rVert }

\author{Simon Bortz}
\author{Moritz Egert}
\author{Olli Saari}

\address{Department of Mathematics, University of Alabama, Tuscaloosa, AL, 35487, USA}
	\email{sbortz@ua.edu} 

\address{Moritz Egert, Department of Mathematics, TU Darmstadt, Schlossgartenstra\ss e 7, 64289 Darmstadt, Germany}
\email{egert@mathematik.tu-darmstadt.de}

\address{Olli Saari, Mathematical Institute, University of Bonn, Endenicher Allee 60, 53115 Bonn, Germany} 
\address{Departament de Matem\`atiques, 
	Universitat Polit\`ecnica de Catalunya,
	Avinguda Diagonal 647, 08028 Barcelona,
	Catalunya, Spain}
\email{olli.saari@upc.edu}


\begin{document}
\allowdisplaybreaks

\title{Carleson Conditions for Weights: The quantitative small constant case}

\begin{abstract}
We investigate the small constant case of a characterization of $A_\infty$ weights due to Fefferman, Kenig and Pipher. In their work, Fefferman, Kenig and Pipher bound the logarithm of the $A_\infty$ constant by the Carleson norm of a measure built out of the heat extension, up to a multiplicative and additive constant (as well as the converse). We prove, qualitatively, that when one of these quantities is small, then so is the other.
In fact,  we show that these quantities are bounded by a constant times the square root of the other, provided at least one of them is sufficiently small.

We also give an application of our result to the study of elliptic measures associated to elliptic operators with coefficients satisfying the ``Dahlberg--Kenig--Pipher" condition. 
\end{abstract}
\maketitle

\section{Introduction}
\noindent The purpose of this work is to study the small constant version of a characterization of $A_\infty$ weights due to Fefferman, Kenig and Pipher \cite{FKP}. The theory of (Muckenhoupt) $A_\infty$ weights, 
initially introduced in the study of maximal functions and singular integral operators, has played a fundamental role in the $L^p$ solvability of elliptic equations. In the paper \cite{FKP}, the authors were investigating the stability of the $A_\infty$ property of elliptic measure under certain perturbations quantified by a Carleson condition. In order to produce a counterexample (to show their main theorem was sharp) they introduced the Carleson characterization of $A_\infty$ weights investigated here. 

Some time later, 
Korey \cite{Koreycarl} studied a qualitative, asymptotic version of this $A_\infty$ characterization.
We also mention the related works on ideal weights \cite{Kor2}, functions of vanishing mean oscillation \cite{Sarason} and continuity of weighted estimates \cite{PV}.
It is however not clear from the arguments neither in \cite{FKP} nor in \cite{Koreycarl},
whether the qualitative, asymptotic  estimates in \cite{Koreycarl} can be made effective or quantified.
Our main result shows that this is the case, 
and we can even assert quantitative (square root) control at the asymptotic limit.
This is the content of the following theorem and its converse, Theorem \ref{converseweight.thm}.
For the Carleson norm $\|\mu_w\|_{\mathcal{C}}$ and the $A_\infty$ characteristic $[w]_{A_\infty}$ see Definitions~\ref{carlesondef.def} and \ref{ainfty.def} below. 
\begin{theorem}\label{mainthrm.thrm}
Let $w$ be a doubling weight on $\mathbb{R}^n$. 
Set $\phi(x) := \pi^{-n/2}e^{-|x|^2}$ and define the measure $\mu_w$ on $(0,\infty) \times \mathbb{R}^n$ by 
\[d\mu_w(r,x) := |\nabla_x \log (\phi_r \ast w)(x)|^2 r \, dx \, dr,\]
where $\phi_r(x) := r^{-n} \phi(x/r)$. 
There exist $\varepsilon > 0$ and $C_1, C_2 > 0$ depending on $n$ and the doubling constant of $w$ such that 
if the Carleson norm of $\mu_w$ satisfies $\|\mu_w\|_{\mathcal{C}} \le \varepsilon$, 
then it holds
\[\log[w]_{A_\infty} \le C_1 \sqrt{\|\mu_w\|_{\mathcal{C}}}, \quad    \no{\log w}_{BMO}   \le C_2\sqrt[4]{\|\mu_w\|_{\mathcal{C}}}.\]
\end{theorem}
\begin{remarks}\label{mtmeasures.rmk}
(1) We suspect the bound $\no{\log w}_{BMO}  \lesssim \sqrt{\|\mu_w\|_{\mathcal{C}}} $ would be the sharp asymptotic power dependence in the second inequality.
We discuss this briefly in Section \ref{closermks.sect}. \\
(2) Theorem \ref{mainthrm.thrm} also holds for doubling measures. 
If $\nu$ is a doubling {\it measure}, 
then one can define $d\mu_\nu(x,r)$ in the same manner and if $\|\mu_\nu\|_{\mathcal{C}} < \infty$, then $\nu$ is absolutely continuous with respect to Lebesgue measure and $d\mu_\nu = d\mu_w$ for $w := \tfrac{d\nu}{dx}$, see \cite[Lemma~2.19]{BTZ2}.
\end{remarks} 

The proof of Theorem \ref{mainthrm.thrm} is contained in Theorem \ref{localoscbycarleson.thm}. 
The main idea is to use a heat flow identity as in \cite{FKP}
to interpret the weight of the Carleson measure as a right hand side of a heat equation.
To make the estimates asymptotically quantitative and effective,
we use the underlying heat equation more efficiently than what was done in \cite{Koreycarl}.
Our proof is short and very easy to read in comparison with the more geometric arguments in \cite{Koreycarl}.

As mentioned above, we also prove a converse of Theorem \ref{mainthrm.thrm} (up to powers), see  Theorem~\ref{converseweight.thm}. The proof there is a little more geometric, but {\it very} different from \cite{Koreycarl}. In particular, two ingredients are central to the proof of Theorem~\ref{converseweight.thm}. First, we use sharp estimates on $A_p$ and reverse-H\"older constants for weights whose logarithm has small BMO norm (Lemma~\ref{korthrm9.lem}). Second, we exploit how such weights act when integrated on two disjoint subsets of a ball with half measure (Corollary~\ref{korstrongdoub.lem}) in order to get cancellation from the fact that $\nabla_x \phi$ is odd.

Our second motivation for reinvestigating the work of Fefferman, Kenig and Pipher 
comes from some recent developments in the study of elliptic equations with variable coefficients 
which satisfy an oscillation condition, 
sometimes known as the Dahlberg--Kenig--Pipher (DKP) condition. 
This condition was first introduced by Dahlberg 
and then studied by Kenig and Pipher \cite{KP} (see also \cite{DPP, HL}). 
It was shown very recently by David, Li and Mayboroda \cite{DLM} 
that the Green functions for such operators in the upper half-space have gradients 
that are quantitatively small in tangential directions. 
The first author, Toro and Zhao \cite{BTZ2} showed 
that these estimates can be used to control the Carleson norm of $\mu_w$ in Theorem \ref{mainthrm.thrm}, 
when $w$ is the elliptic measure. 
Combining Theorem~\ref{mainthrm.thrm} and the results in \cite{BTZ2}, 
we prove that if the constant in the (weak) DKP condition is small, then the elliptic measure has close to optimal $A_\infty$ constant (Theorem~\ref{dkpsmallthrm.thrm}). 
This quantitative control is what is needed in order to push the small/vanishing constant theory to rougher settings\footnote{This was indeed done in \cite{DLM3} after the first version of the present article was posted.} (Lipschitz graphs and chord arc domains). 
Further details are provided at the end of Section \ref{DKPsect}.  

We point out that a dyadic version of Theorem \ref{mainthrm.thrm} 
can be deduced from the work of Fefferman, Kenig and Pipher \cite[Theorem 3.22]{FKP} 
(see also\footnote{There are differences in the proofs in \cite{bucksum} and \cite{FKP}, and \cite{bucksum} gives a comprehensive study of dyadic $A_\infty$ and summation conditions similar to those in \cite{FKP}.} \cite[Theorem 2.2 (iii)]{bucksum}). 
While the dyadic arguments include the (implicit or explicit) use of martingales,
a special role in the computations for the continuous setting is played by the heat kernel.
Interestingly, the Fefferman--Stein change of kernel argument shows 
that the heat kernel itself is not important in the Carleson condition. 
While the present work provides a very concrete and quantitative approach 
for a Carleson condition based on the heat kernel,
we still need the change of kernel argument to deal with more general kernels that are needed for the applications,
such as Theorem \ref{dkpsmallthrm.thrm}. 
It is unclear to us if there is a way to treat the setting with dyadic martingales 
and the continuous setting with general kernels in a unified way
without resorting to the change of kernel argument.

\subsection*{Acknowledgment}
This work was supported by the Deutsche Forschungsgemeinschaft 
(DFG, German Research Foundation) under Germany's Excellence Strategy -- EXC-2047/1 -- 390685813 and CRC 1060, by the ANR project RAGE ANR-18-CE40-0012 and by `Verein von Freunden der Technischen Universit\"at zu Darmstadt e.V.'. The first author would like to thank Tatiana Toro and Zihui Zhao for some helpful comments about the paper.

\section{Notation}
\noindent 
We fix a dimension $n \in \mathbb{N}$ throughout. 
In what follows, $w$ will always denote a weight, 
that is, a non-negative locally integrable function. 
As is customary in the study of weights, 
we will abuse notation and write $w$ also for the measure $w \, dx$ 
so that for a Borel set $E$ we have $w(E) = \int_E w(x) \, dx$. 
We denote the Lebesgue measure of $E$ by $|E|$.
Given a locally integrable function $f$,
we use the notation
\[
f_E = \fint_{E} f(x) \,dx := \frac{1}{|E|}\int_{E} f(x) \,dx
\]
for the mean value of $w$ over the set $E$.
Throughout, we set $\Delta(x,r)$ to be the open Euclidean ball in $\rn$ with center $x \in \rn$ and radius $r > 0$. 
If $\Delta(x,r) \subset \rn$ is a ball,
we associate with it the Carleson box 
\[
T(x,r) := T(\Delta(x,r)) :=  (0,r) \times \Delta(x,r)
\] 
in $\ree_+ = (0,\infty) \times \R^n$. For any function $f: \rn \to \re$ and any number $r > 0$,
we denote
\[f_r(x) := r^{-n} f(x/r).\]
We reserve the symbol $\phi$ for the heat kernel
\[\phi(x) := \pi^{-n/2}e^{-|x|^2}.\]

\begin{definition}[Doubling]
\label{doubling.def}
Given a Borel measure $\nu$ on $\rn$, 
we say $\nu$ is doubling if $\nu$ is non-trivial and there exists a constant $C_{\nu}$ such that 
\[\nu(\Delta(x,2r)) \le C_{\nu} \nu(\Delta(x,r)), \quad \forall x \in \rn, r > 0.\]
We call the least $C_{\nu}$ the doubling constant for $\nu$. 
\end{definition}

\begin{definition}[$A_{\infty}$ weights]
\label{ainfty.def}
Given a weight function $w$,
we define the $A_{\infty}$ constant as 
\[
[w]_{A_{\infty}} := \sup_{\Delta} w_\Delta \exp\left( -\fint_{\Delta} \log w(x)\, dx \right),
\]
where the supremum is taken over all Euclidean balls $\Delta$.
\end{definition}

We use the original definition of Hru\v{s}\v{c}ev of the $A_{\infty}$ constant,
but we remark that the choice of the definition is not unique \cite{MR3473651}
and that the different definitions of $A_{\infty}$ constants lead to very different values in general \cite{MR3092729}.
Although for instance \cite{parissisrela} prefers to study the so-called Fujii--Wilson constant
for asymptotically sharp reverse H\"older inequalities,
the constant from Definition \ref{ainfty.def} seems to arise very naturally in our considerations.
It remains, however, an interesting question 
if results similar to Theorem \ref{mainthrm.thrm} and Theorem \ref{converseweight.thm} hold for other $A_{\infty}$ constants.

\begin{definition}
\label{bmo.def}
Given a locally integrable function $f$,
we define its $\BMO$ norm as 
\[
\no{f}_{\BMO} :=  \sup_{\Delta} \fint_{\Delta} |f(y) - f_{\Delta} | \, dy,
\]
where the supremum is taken over all Euclidean balls $\Delta$.
\end{definition}

\begin{remark}
If $w$ is a weight and $\log w$ is in $\BMO$,
then it follows from the John--Nirenberg inequality that
$[w]_{A_{\infty}} \le 1 + C \no{\log w}_{\BMO}$ with $C$ only depending on the dimension, provided $\no{\log w}_{\BMO}$ is sufficiently small (depending on dimension).
Therefore we could drop the assumption that $w \in A_\infty$ in Subsection~\ref{smconstainfyerrs.sect}, but we prefer to keep it for the sake of clarity.
\end{remark}

\begin{definition}[Carleson measure]
\label{carlesondef.def}
Given a Borel measure $\nu$ on $\ree_+$,
we define its Carleson norm as 
\begin{equation} 
\|\nu\|_{\mathcal{C}}  := \sup_{\Delta} \frac{\nu(T(\Delta))}{|\Delta|},
\end{equation} 
where the supremum is taken over all Euclidean balls $\Delta$.
\end{definition}

\section{Quantitative estimates on weights}
\label{weights.sec}
\noindent
Let $w$ be a doubling weight in the sense of Definition \ref{doubling.def}.
Denote by $u_w(r,x) := \phi_{\sqrt{r}}*w(x)$ the heat extension of the weight $w$, so that $\partial_r u_w - \frac{1}{4} \Delta_x u_w = 0$. Define the measure $\mu_w$ on $\ree_+$ by
\[d\mu_w(r,x) := |\nabla_x \log u_w(r^{2},x)|^2 r \, dx dr.\]
In the following, we abbreviate $u = u_w$, $\mu = \mu_w$, $\nabla = \nabla_x$, $\Delta = \Delta_x$. Also, we use the above relations between the symbols $w$, $u$ and $\mu$ without further mention in this section. Thanks to doubling of $w$, a dyadic decomposition of the convolution integral yields comparability of $u(r^2,x)$ and $w(\Delta(x,r))$:
\begin{align}
\label{u_vs_measure.eq}
0<\frac{1}{\pi^{n/2} e} \leq \frac{u(r^2,x)}{w(\Delta(x,r))} \leq \sum_{j=1}^\infty \frac{C_w^j} {\pi^{n/2} e^{4^{j-1}}} < \infty.
\end{align}
The characterization of the $A_{\infty}$ class by Fefferman--Kenig--Pipher \cite{FKP} 
is explained by the heat flow identity 
\begin{equation}
\label{FKPcomputation.eq}
\partial_s \log u(s,y) - \frac{1}{4} \Delta \log u(s,y) = \frac{1}{4} \abs{ \nabla \log u(s,y) }^{2},
\end{equation} 
where the right hand side corresponds to the relevant Carleson measure.
We need precise estimates for contributions arising from the three terms in \eqref{FKPcomputation.eq}.
These are proved in the following two subsections, 
first in terms of small Carleson constants and then in terms of small $A_{\infty}$ constants.

\subsection{Case of small Carleson norm}
We start with a pointwise estimate that is needed in the proof.
For any $x \in \rn$ and $r > 0$,
denote $W(x,r) := (r/2,r) \times \Delta(x,r)$. 

\begin{lemma}
\label{moserconsequence.lem} 
Let $w$ be a doubling weight.
There exists a constant $C$ depending only on dimension
and the doubling constant such that
for all $(r,x) \in \ree_+$,
\[ \frac{r^{2}|\Delta u(r^{2},x)|}{u(r^{2},x)} +  \frac{r |\nabla u(r^{2},x)|}{u(r^{2},x)}  \le 
C \sqrt{ \frac{1}{|\Delta(x,r)|} \iint_{T(x,r)} \frac{|\nabla u(t^{2},y)|^{2}}{u(t^{2},y)^{2}} t \, dy dt }.\]
\end{lemma}
\begin{proof}
    Fix $x \in \rn$ and $r > 0$.
    Let $k \in \{1,2\}$.
    Denote $W := (r/2,r) \times \Delta(x,r)$ and $\widetilde{W} :=(3r/4,r) \times \Delta(x,r/2)$.
    By the local estimate in \cite[Section~2.3c]{Evans} we have
\begin{align*}
\frac{1}{r^{2(k-1)+1}|W|} \iint_{W} |\nabla u(t^{2},y)|^{2} t \, dydt 
&= \frac{1}{r^{2k} |\Delta(x,r)| } \iint_{(r^{2}/4,r^{2}) \times \Delta(x,r)} |\nabla u(t,y)|^{2} \, dydt \\
&\ge C \sup_{(t,y) \in \widetilde{W}} |\nabla^{k} u(t^{2},y)|^{2}.
\end{align*}
Using the parabolic Harnack inequality~\cite[Section 7.1b]{Evans} and then \eqref{u_vs_measure.eq} together with the doubling property of the weight, we obtain
\[
\sup_{(t,y) \in W} u(t^{2},y) \le C \inf_{y \in \Delta(x,r)} u(2r^{2},y) \le  C \inf_{(t,y) \in \widetilde{W}} u(t^{2},y).
\]
Together, these two imply 
\[
\frac{1}{|\Delta(x,r)|} \iint_{W} \frac{|\nabla u(t^{2},y)|^{2}}{u(t^{2},y)^{2}} t \, dy dt
    \ge C r^{2k} \sup_{(t,y) \in \widetilde{W}} \frac{|\nabla^{k} u(t^{2},y)|^{2}}{u(t^{2},y)^{2}}
\]
and in particular, by smoothness of $u$ and since $(r,x)$ is on the boundary of $\widetilde{W}$,
\[
r^{k} \frac{|\nabla^{k}u(r^{2},x)|}{u(r^{2},x)} = \sqrt{\frac{r^{2k}|\nabla u(r^{2},x)|^{2}}{u(r^{2},x)^{2}}}
    \le C \sqrt{ \frac{1}{|\Delta(x,r)|} \iint_{W} \frac{|\nabla u(t^{2},y)|^{2}}{u(t^{2},y)^{2}} t \, dy dt },
\]
which is the claimed inequality.
\end{proof}

\begin{theorem}
\label{localoscbycarleson.thm}
Let $\Delta(x,r) \subset \rn$ be a ball
and $w$ a doubling weight.
There exist constants $C \ge 1$ and $\varepsilon > 0$ only depending on the dimension and the doubling constant such that if 
$\sqrt{\no{ 1_{T(x,2r)} \mu_w }_{\mathcal{C}}} \le \varepsilon$,
then
\begin{align}
\label{eq:localoscbycarleson-1}
\log \fint_{\Delta(x,r)} w(y) \, dy - \fint_{\Delta(x,r)} \log w(y) \, dy  &\le C  \sqrt{\no{ 1_{T(x,2r)} \mu_w }_{\mathcal{C}}}, \\
\label{eq:localoscbycarleson-2}
\fint_{\Delta(x,r)} |\log(y) - (\log)_{\Delta(x,r)}| \, dy &\le C \sqrt[4]{\no{ 1_{T(x,2r)} \mu_w }_{\mathcal{C}}}.
\end{align}
In particular, if $\sqrt{\no{\mu_w}_{\mathcal{C}}} \le \varepsilon$, 
then 
\[
\log [w]_{A_{\infty}} \le C \sqrt{\no{\mu_w}_{\mathcal{C}}}, \quad \no{\log w}_{BMO} \le C \sqrt[4]{\no{\mu_w}_{\mathcal{C}}}.
\]
\end{theorem}
\begin{proof}
Without loss of generality, 
we assume $x=0$ and $r=1$.
The claim then follows by scaling and translation.
We start by proving \eqref{eq:localoscbycarleson-1}.
We denote $\Delta(0,1)= \Delta$ for brevity.
We start by writing 
\[
0 \le \log w_{\Delta} - (\log w)_{\Delta}
    = (\log w_{\Delta} -  (\log u(1,\cdot ))_{\Delta} ) + ((\log u(1,\cdot ))_{\Delta} - (\log w)_{\Delta})
    = \I + \II .
\]
By the fundamental theorem of calculus and
the heat flow identity, 
\[
\II = \fint_{\Delta} \int_{0}^{1} \partial_t \log u(t,y) \, dtdy
    = \frac{1}{4|\Delta|} \iint_{T(\Delta)} \left( \Delta \log u(t,y)  +  |\nabla \log u (t,y)|^{2} \right) \, dtdy .
\]
By definition,
\begin{multline*}
\frac{1}{4|\Delta|}\iint_{T(\Delta)} |\nabla \log u (t,y)|^{2} \, dtdy =\frac{1}{8|\Delta|}\iint_{T(\Delta)}  |\nabla \log u (t^{2},y)|^{2} \, t dtdy \\
 \le \frac{1}{8} \no{ 1_{T(0,2)} \mu_w }_{\mathcal{C}} .
\end{multline*}
By the divergence theorem and Lemma \ref{moserconsequence.lem},
\[
\frac{1}{4|\Delta|} \iint_{T(\Delta)}   \Delta \log u(t,y) \le \frac{1}{4|\Delta|} \int_{\partial \Delta} \int_{0}^{1} |\nabla \log u(t,y)| \le C \sqrt{\no{ 1_{T(0,2)} \mu_w }_{\mathcal{C}}}.
\]
This concludes the estimation of the term $\II$.

For the other term $\I$,
we write 
\begin{equation}
\label{eq:I-small-ainfty}
\I = (\log w_{\Delta} -  \sup_{y \in \Delta} \log u(1,y)    )  +  (   \sup_{y \in \Delta} \log u(1,y) -  (\log u(1,\cdot ))_{\Delta} ) .
\end{equation}
By the fundamental theorem of calculus and Lemma \ref{moserconsequence.lem}, 
\[
 \sup_{y \in \Delta} \log u(1,y) -  (\log u(1,\cdot ))_{\Delta}  \le 2 \sup_{y \in \Delta} |\nabla \log u(1,y)| \le C \sqrt{\no{ 1_{T(0,2)} \mu_w }_{\mathcal{C}}}.
\]
To bound the other term in \eqref{eq:I-small-ainfty},
we compute 
\[
(\log w_{\Delta} -  \sup_{y \in \Delta} \log u(1,y)    )
\le  \log_{+}  \left( \frac{w_{\Delta}- \sup_{y \in \Delta} u(1,y)}{\sup_{y \in \Delta} u(1,y)} + 1  \right),
\]
where $\log_{+} s = \log  \max\{s,1\} $.
Estimating the supremum in the numerator by the average,
applying the fundamental theorem of calculus, 
using that $u$ is a solution to the heat equation
and using the divergence theorem,
we obtain
\[
\frac{w_{\Delta}- \sup_{y \in \Delta} u(1,y)}{\sup_{y \in \Delta} u(1,y)} 
\le \frac{1}{\sup_{y \in \Delta} u(1,y)}  \frac{1}{|\Delta|} \int_{ \partial \Delta} \int _{0}^{1} |\nabla u(t,y) | \, dtdy .
\]
This is bounded by 
\begin{multline}
\label{eq:sum-in-correction}
\sum_{k=-\infty}^{0} \sup_{2^{k-1} \le s \le 2^{k} } \sup_{y \in \partial \Delta} \frac{u(s,y)}{ u(1,y)}  \frac{1}{|\Delta|} \int_{ \partial \Delta} \int _{2^{k-1}}^{2^{k}} \frac{|\nabla u(t,y)|}{u(t,y)} \, dtdy  \\
 \le C \sqrt{\no{ 1_{T(0,2)} \mu_w }_{\mathcal{C}}} \sum_{k=-\infty}^{0} 2^{k/2} \sup_{2^{k-1} \le s \le 2^{k} } \sup_{y \in \partial \Delta} \frac{u(s,y)}{ u(1,y)}, 
\end{multline}
where the last inequality is due to Lemma \ref{moserconsequence.lem}.
Here,
\[
\frac{u(s,y)}{ u(1,y)}
= \exp \left( \log u(s,y)  -  \log u(1,y)  \right)
= \exp \left( - \int_{s}^{1} \partial_t \log u (t,y) \, dt \right) 
\]
and by the heat equation and Lemma \ref{moserconsequence.lem},
\[
- \int_{s}^{1} \partial_t \log u (t,y) \, dt
= - \frac{1}{4} \int_{s}^{1}  \frac{\Delta u (t,y)}{u(t,y)}  \, dt  
\le C \sqrt{\no{ 1_{T(0,2)} \mu_w }_{\mathcal{C}}} \log \frac{1}{s} .
\]
Exponentiating,
we conclude 
\[
\frac{u(s,y)}{ u(1,y)} \le s^{-C \sqrt{\no{ 1_{T(0,2)} \mu_w }_{\mathcal{C}}} } ,
\]
and provided $C \sqrt{\no{ 1_{T(0,2)} \mu_w }_{\mathcal{C}}} < 1 / 4$,
the sum in \eqref{eq:sum-in-correction} converges to an absolute constant and the proof of \eqref{eq:localoscbycarleson-1} is complete.
The second inequality \eqref{eq:localoscbycarleson-2} follows from \eqref{eq:localoscbycarleson-1} by Theorem 6 in Section 3.5 of \cite{Kor2}. 
\end{proof}

\begin{remark}
\label{rmk:bad-carleson}
One may argue similarly to the Theorem \ref{localoscbycarleson.thm} to obtain for all balls $\Delta(x,r) \subset \rn$
\begin{multline*}
\fint_{\Delta(x,r)}  | \log w (y) - (\log w)_{\Delta(x,r)}| \, dy\\
\le C\left( \sqrt{\no{ 1_{T(x,2r)} \mu_w }_{\mathcal{C}}} + \sup_{\Delta' \subset \Delta(x,2r)}  \frac{1}{|\Delta'|} \iint_{T(\Delta')} \frac{|\nabla u(t^{2},y)|}{u(t^{2},y)} dtdy \right).
\end{multline*} 
Such an estimate does not imply the conclusion of Theorem \ref{localoscbycarleson.thm}
but it could be used for instance in the setting of H\"older continuous coefficient matrix in our application (Theorem \ref{dkpsmallthrm.thrm}), see the remarks in Section \ref{closermks.sect}.
\end{remark}


\subsection{Case of small weight constant}\label{smconstainfyerrs.sect}

Next we prove an analog of Lemma \ref{moserconsequence.lem},
but now we seek a right hand side with $A_{\infty}$ constant instead of the Carleson norm.
We start with an auxiliary result on $A_{p}$ weights. 
Given $p > 1$,
we denote 
\[
[w]_{A_{p}} := \sup_{\Delta} w_{\Delta} \left( \fint_{\Delta} w(y)^{1/(1-p)} \, dy \right)^{p-1}, \quad  [w]_{B_{p}} := \sup_{\Delta} \frac{1}{w_\Delta} \left( \fint_\Delta w(y)^{p} \, dy \right)^{1/p}.
\]  
Note that by Jensen's inequality $1 \leq [w]_{A_{\infty}} \le [w]_{A_{p}}$. The following has been proved in \cite[Theorem 9]{Kor2}. We state the result in terms of $\no{\log w}_{\BMO}$ instead of $[w]_{A_{\infty}}$, which is an intermediate step in the proof of \cite[Theorem 9]{Kor2}. In fact, the author is only using the smallness of the BMO norm provided by the (small) $A_{\infty}$ constant.

\begin{lemma}\label{korthrm9.lem}
There exist constants $\varepsilon_0 > 0$ and $K > 0$ depending only on dimension 
such that the following holds.
Let $w \in A_\infty$ with 
$\no{\log w}_{\BMO} \leq \varepsilon$
for some $0< \varepsilon < \varepsilon_0$. 
Then for $p = 1 + K  \varepsilon $ it holds 
\[
\max\{[w]_{A_p}, [w]_{B_{1/(p-1)}}\} \le p .
\]
In particular, since $[w]_{A_\infty} \le [w]_{A_p}$, it follows that $[w]_{A_\infty} \le 1 + K\varepsilon$.
\end{lemma}

Lemma \ref{korthrm9.lem} is used in conjunction with the following two corollaries.
It helps to give precise estimates for ratios of weighted measures of sets.

\begin{corollary}\label{alopquotients.cor}
There exist constants $\varepsilon_0 > 0$ and $K > 0$ depending only on dimension 
such that if $w \in A_\infty$ with $ \no{\log w}_{\BMO}  \leq \varepsilon$ for some $0< \varepsilon < \varepsilon_0$,  
then for any ball $\Delta$ and any measurable subset $E \subseteq \Delta$  
\begin{equation}\label{aloptfromBq.eq}
\frac{w(E)}{w(\Delta)} \le (1 + K \varepsilon) \left(\frac{|E|}{|\Delta|} \right)^{1 -K \varepsilon}
\end{equation}
and   
\begin{equation}\label{aloptfromAp.eq}
\left( \frac{|E|}{|\Delta|} \right)^{1+K \varepsilon} \le (1 + K \varepsilon)  \frac{w(E)}{w(\Delta)}  .
\end{equation}
\end{corollary}
\begin{proof}
The claim follows by H\"older's inequality and the definitions.
More precisely,
inequality \eqref{aloptfromBq.eq} takes advantage of the $B_{1/(p-1)}$ condition 
whereas inequality \eqref{aloptfromAp.eq} uses the $A_{p}$ condition.
The (standard) computation is left for the reader.
\end{proof}

\begin{corollary}
\label{korstrongdoub.lem}
There exists $\varepsilon_0 > 0$ and a constant $C$ only depending on the dimension such that the following holds. 
If $w \in A_\infty$ with $ \no{\log w}_{\BMO}  \leq \varepsilon$ for some $0< \varepsilon < \varepsilon_0$,
then for all balls $\Delta$ and all measurable subsets $E,F \subset \Delta$ with
\[
E \cup F = \Delta, \quad |E\cap F| = 0 , \quad  |E| = |F|
\]
it holds
\[\frac{w(E)}{w(F)} \le 1 + C\varepsilon.\]
\end{corollary}
\begin{proof}
Let $\varepsilon_0$ and $K$ be as in Corollary \ref{alopquotients.cor}.
Making $\varepsilon_0$ smaller if needed,
we can assume $K \varepsilon_0 \le 1/2$.
Let $\varepsilon < \varepsilon_0$ be positive.
By Corollary \ref{alopquotients.cor},
\[
\frac{w(E)}{w(F)} = \frac{w(E)}{w(\Delta)} \frac{w(\Delta)}{w(F)} \le (1 + K\varepsilon)^{2} 2^{2K \varepsilon},
\]
where we estimated the first ratio by \eqref{aloptfromBq.eq} and the second ratio by \eqref{aloptfromAp.eq}.
The claim then follows by Taylor's theorem.
\end{proof}

From now on, $\eps_0$ is fixed such that the three Lemmas \ref{korthrm9.lem}, \ref{alopquotients.cor} and \ref{korstrongdoub.lem} all hold.

\begin{lemma}
\label{nablaoverphi.lem}
Let $\varepsilon_0 > 0$ be as above. There exists $C$ only depending on the dimension such that the following holds.
If $w \in A_\infty$ is a weight with $ \no{\log w}_{\BMO}  \leq \varepsilon$ for some $0< \varepsilon < \varepsilon_0$, then for all $r>0$ and $x \in \R^n$,
\[
\frac{|\nabla (\phi_r * w )(x)|}{(\phi_r * w)(x)} \le \frac{C \varepsilon}{r}.
\]
\end{lemma}
\begin{proof}
Since $\nabla(\phi_r * w )(x) = r^{-1} ((\nabla \phi)_r * w )(x)$, we see by translation and scaling invariance of the assumption on $w$ that it suffices to prove the claim for $r=1$ and $x = 0$.
Clearly, for instance by Corollary \ref{alopquotients.cor},
$w$ is a doubling measure with an absolute doubling constant $D$.
Since 
\[
(\phi * w)(0) \ge \pi^{-n/2} e^{-1} w(\Delta(0,1)),
\]
it will be enough to prove
\begin{equation}
	\label{goal_old_thm.eq}
	|( \nabla \phi * w )(0)| \le C \varepsilon \, w(\Delta(0,1)).
\end{equation}
Recall that $\nabla \phi(x) = - (x/|x|) \cdot 2 \pi^{-n/2} |x| e^{-|x|^{2}} $. As this function is odd,
it holds 
\[
|(\nabla \phi * w )(0)| =    \abs{\int_{\R^n} \nabla \phi(y)  \frac{w(y) - w(-y)}{2} \, dy }.
\]
Splitting the domain of integration into dyadic annuli, $C_0 := \Delta(0,2)$, $C_j := \Delta(0,2^{j+1}) \setminus \Delta(0,2^j)$ for $j \geq 1$,
and bounding $|\nabla \phi|$ by its maximum on each set yields
\begin{align}
\label{shellsproof1.eq}
	|(\nabla \phi * w )(0)|
	& \leq \sum_{j=0}^\infty 4 \pi^{-n/2} 2^j e^{-4^j} \int_{\Delta(0,2^{j+1})} \abs{\frac{w(y) - w(-y)}{2}} \, dy.
\end{align}
Fixing $j$ for a moment,
we denote
\begin{align*}
X_{j} &:= \{ y \in \Delta(0,2^{j+1}) : w(y) - w(-y) > 0  \},\\
Y_{j} &:= \{ y \in \Delta(0,2^{j+1}) : w(y) - w(-y) < 0  \},\\
Z_{j}' &:= \{ y \in \Delta(0,2^{j+1}) : w(y) - w(-y) = 0  \}.
\end{align*}
It holds $|X_{j}| = |Y_{j}|$ and $|Z_{j}' \cap H| = |Z_{j}'|/2$ for any half space $H$ with $0 \in \partial H$.
We fix one half space $H$ and denote $Z_{j} := H \cap Z_{j}'$.
Introducing $E_{j} := X_{j} \cup Z_{j}$,
we obtain essentially disjoint sets $E_{j}$ and $-E_{j}$ such that $E_{j} \cup( -E_{j}) = \Delta(0, 2^{j+1})$
and $|E_{j}| =|-E_{j}|$. This allows us to bring Lemma~\ref{korstrongdoub.lem} into play. We get
\begin{align}
\begin{split}
\label{shellsproof2.eq}
\int_{\Delta(0,2^{j+1})} &\abs{\frac{w(y) - w(-y)}{2}} \, dy
= 2 \int_{E_j} \abs{\frac{w(y) - w(-y)}{2}} \, dy \\
&=  w(E_j) - w(-E_j) 
= \left(\frac{w(E_j)}{w(-E_j)} - 1 \right)w(-E_j)
\le C \varepsilon \, w(-E_j),
\end{split}
\end{align}
where the second equality used that $w(y) - w(-y) \geq 0$ in $E_{j}$. Noting that
\[
w(-E_j) \leq w(\Delta(0,2^{j+1})) \leq D^{j+1} w(\Delta(0,1))
\]
by doubling, we conclude the proof of \eqref{goal_old_thm.eq} by combining \eqref{shellsproof1.eq} and \eqref{shellsproof2.eq}.
\end{proof}

In addition to Lemma \ref{nablaoverphi.lem}, we need an estimate to control the ratio of ordinary averages and heat averages.

\begin{lemma}
\label{heatoverrough.lem}
Let $\varepsilon_0$ be as before and $w$ an $A_\infty$ weight with $\no{\log w}_{\BMO}  \leq \varepsilon$ for some $\varepsilon < \varepsilon_0$.
Let $r>0$, $y \in \rn$ and $x \in \Delta(y,r)$.
Then 
\begin{equation}\label{eq0001.eq}
|\log (\phi_r \ast w)(x) - \log w_{\Delta(y,r)}| \le C \varepsilon,
\end{equation}
where $C$ only depends on the dimension.
\end{lemma}
\begin{proof}
As in the proof of Corollary \ref{korstrongdoub.lem},
we may use Corollary \ref{alopquotients.cor}
to conclude that  
\[
|\log w_{\Delta(y,r)} - \log w_{\Delta(x,r)}| 
\le \abs{ \log \frac{w(\Delta(y,r))}{w(\Delta(y,2r))} \frac{w(\Delta(y,2r))}{w(\Delta(x,r))} }
\le C \varepsilon.
\]
Hence, it suffices to prove the claim for $x=y$.
By translation invariance, we may assume that $x=y=0$  and then by scaling invariance we may also assume $r=1$.

We denote 
\begin{equation}
\label{iota_gamma.eq}
\iota (\lambda) :=  (- \log (\lambda \pi^{n/2}))^{1/2},
\quad \gamma_s(t) := \frac{1}{\pi^{n/2}} \exp(-t^{2/s}),
\end{equation}
where $\lambda \in (0,\pi^{-n/2})$ and $s \in (0,\infty)$.
Note that $\iota^{s}$ is the inverse function of $\gamma_s$.
We write by the Cavalieri principle
\begin{equation}\label{eq0002.eq}
\log (\phi \ast w)(0) - \log w_{\Delta(0,1)}
= \log \left( |\Delta(0,1)| \int_{0}^{\pi^{-n/2}} \frac{w(\Delta(0,\iota(\lambda)))}{w(\Delta(0,1))} \, d \lambda \right) .
\end{equation}
If $\iota(\lambda) \ge 1$,
we apply \eqref{aloptfromAp.eq} to estimate 
\[
\frac{w(\Delta(0,\iota(\lambda)))}{w(\Delta(0,1))} \le (1+K \varepsilon) \iota(\lambda)^{(1+K\varepsilon)n}.
\]
If $\iota (\lambda) < 1$, 
we apply \eqref{aloptfromBq.eq} to estimate 
\[
\frac{w(\Delta(0,\iota(\lambda)))}{w(\Delta(0,1))} \le (1+K \varepsilon) \iota(\lambda)^{(1-K\varepsilon)n}.
\]
In order to match the exponents of $\iota(\lambda)$ in the estimates above,
we note that for all $0 < a \le  1$ and $s \geq 1/2$,
\[
|a^{(1-K\varepsilon)n} -  a^{(1+K\varepsilon)n}|
\le \int_{(1-K\varepsilon)n}^{(1+K\varepsilon)n}|(\log a) a^{s}| \, ds 
\le C \varepsilon
\]
with $C$ only depending on the dimension. We have $K \eps_0 \leq 1/2$, see the proof of Corollary~\ref{korstrongdoub.lem}. Hence, we can use this estimate with $a = \iota (\lambda)$ when $\iota(\lambda) < 1$ and conclude 
\begin{align*}
 \int_{0}^{\pi^{-n/2}} \frac{w(\Delta(0,\iota(\lambda)))}{w(\Delta(0,1))} \, d \lambda
    &\le C \varepsilon + (1+K \varepsilon) \int_{0}^{\pi^{-n/2}} \iota(\lambda)^{(1+K\varepsilon)n} \, d \lambda \\
    &= C \varepsilon + \frac{1+K \varepsilon}{|\Delta(0,1)|} \int_{\R^n} \gamma_{1+K\varepsilon}(|y|) \, dy,
\end{align*}  
where $\gamma_{1+K\varepsilon}$ is as defined in \eqref{iota_gamma.eq}
and the last equality follows by the Cavalieri principle. 
By the fundamental theorem of calculus,
\[
\int_{\R^n} \gamma_{1+K\varepsilon}(|y|) \, dy = 1 + \int_{\R^n} \int_1^{1+K\eps} \partial_s \gamma_{s}(|y|) \, ds dy \le 1 + C \varepsilon.
\]
The previous two displayed estimates yield an upper bound by $1+ C \eps$ for the term inside the logarithm on the right of \eqref{eq0002.eq}. 

In order to obtain a lower bound of the same type, we would simply reverse the roles of \eqref{aloptfromAp.eq} and \eqref{aloptfromBq.eq} at the beginning of the proof and then repeat the argument up to the obvious changes. Thus, 
\[
|\log (\phi \ast w)(0) - \log w_{\Delta(0,1)}| \leq \log (1+ C \eps) 
\]
and the claimed inequality \eqref{eq0001.eq} (when $x = y = 0$) follows by applying Taylor's theorem.
\end{proof}

\begin{theorem}
\label{converseweight.thm}
There exists $\varepsilon > 0$ depending on dimension such that if 
$w$ is an $A_\infty$ weight with $\no{\log w}_{\BMO} \le \varepsilon$,
then
\begin{equation}\label{convweighthrmeq1.eq}
\no{\mu_w}_{\mathcal{C}} \le C'_1 \no{\log w}_{\BMO} \le C'_2 \sqrt{ \log [w]_{A_{\infty}}},
\end{equation}
where the constants $C'_1,C'_2 \ge 0$ depend on the dimension\footnote{In contrast with Theorem \ref{mainthrm.thrm}, we can control the doubling constant using the $A_\infty$ constant when $\no{\log w}_{\BMO}$ is small.}. 
\end{theorem}

\begin{proof} 
We let $0 < \eps < \eps_0$ as before.
By the usual scaling, it suffices to check the Carleson condition on the unit tent $T(0,1) = (0,1)\times \Delta(0,1)$. Recalling the heat flow identity \eqref{FKPcomputation.eq},
we can use the fundamental theorem of calculus and the divergence theorem as in the proof of Theorem~\ref{localoscbycarleson.thm} to conclude    
\begin{multline*}
\frac{1}{4} \iint_{T(0,1)} \abs{ \nabla \log u(s,y) }^{2}  \, dyds \\
    \le \abs{\int_{\Delta(0,1)} \log \frac{u(1,y)}{w(y)} \, dy} +  C \int_{0}^{1} \int_{\partial \Delta(0,1)} |\nabla\log u(s,y)| \, dH^{n-1}(y) ds
    =: \I + \II.
\end{multline*}
Let us recall that $u(r,y) = (\phi_{\sqrt{r}} * w)(y)$. Thus, we control $\I$ by
\begin{align*}
\I & \leq \abs{\fint_{\Delta(0,1)} (\log (\phi * w)(y) - \log w_{\Delta(0,1)} )\, dy} + \abs{\fint_{\Delta(0,1)}   (\log w_{\Delta(0,1)} - \log w(y) )\, dy} \\
   &\le C \no{\log w}_{\BMO} + \log [w]_{A_{\infty}},
\end{align*}
where we used Lemma \ref{heatoverrough.lem} for the first and the definition of $[w]_{A_{\infty}}$ for the second integral.
To estimate $\II$, we simply apply Lemma \ref{nablaoverphi.lem}. 
We have shown 
\[(\I+\II)/|\Delta(0,1)| \le C ( \no{\log w}_{\BMO}  + \log [w]_{A_{\infty}}),\] 
which, upon applying Lemma \ref{korthrm9.lem}, yields the first inequality in \eqref{convweighthrmeq1.eq}.

The second inequality in \eqref{convweighthrmeq1.eq} is a direct consequence of the first and \cite[Theorem 4]{Kor2}, which states that $\no{\log w}_{\BMO} \leq C \sqrt{\log[w]_{A_\infty}}$ if $[w]_{A_\infty} \leq 2$.
\end{proof}

\begin{remark}
Again by \cite[Theorem 4]{Kor2} the \emph{a priori} smallness in Theorem~\ref{converseweight.thm} can equivalently be assumed for $\log[w]_{A_\infty}$.
\end{remark}

\section{An application to elliptic equations satisfying the Dahlberg--Kenig--Pipher condition.}
\label{DKPsect}

\noindent Here, we apply our quantitative estimates to the study of elliptic measures for elliptic operators satisfying the Dahlberg--Kenig--Pipher (DKP) condition. We need to give several definitions before stating our application. The reader may also wish to consult \cite[Section 2.4]{AGMT} and \cite{BTZ2}.

\begin{definition}[Elliptic matrices and operators]
Fix $\Lambda \ge 1$. We say a matrix-valued function $A: \ree_+ \to M_{n+1}(\re)$ is $\Lambda$-elliptic if $\|A\|_{L^\infty(\ree_+)} \le \Lambda$ and
\[\langle A(X) \xi, \xi \rangle \ge \Lambda^{-1} |\xi|^2, \quad \forall \xi \in \ree, X \in \ree_+.\]
We say $A$ is elliptic if it is $\Lambda$-elliptic for some $\Lambda \ge 1$. 
The smallest such $\Lambda$ is called the ellipticity constant of $A$.
We say $L$ is a divergence form elliptic operator (on $\ree_+$) if $L = -\div A \nabla$ (viewed in the weak sense) for an elliptic matrix $A$.
We denote by $A^{T}$ the transpose of $A$ and by $L^{T}$ the associated operator.
Given $\Omega \subseteq \ree_+$ open, we say $u \in W^{1,2}_{loc}(\om)$ is a weak solution to $Lu = 0$ in $\om$ if 
\[\iint_\om A \nabla u \cdot \nabla F \, dX = 0, \quad \forall F \in C_c^\infty(\om).\]
\end{definition}

For every elliptic operator there is an associated family of elliptic measures. 
\begin{definition}[Elliptic measure]\label{def:Green}
Let $L = -\div A \nabla$ be a divergence form elliptic operator on $\ree_+$. There exists a family of Borel probability measures on $\rn$, $\{\hm^X_L\}_{X \in \ree_+}$, such that for $f \in C_c^\infty(\rn)$ the function
\[u(X) = \int_{\rn} f(y) \, d\hm_L^X(y)\]
is the unique weak solution to the Dirichlet problem
\[(D)_L \quad \begin{cases}
Lu &= 0 \text{ in } \ree_+, \\
u|_{\{0\} \times \rn} &= f,
\end{cases}\]
satisfying $u \in C(\overline{\ree_+})$ and $u(X) \to 0$ as $|X| \to \infty$ in $\ree_+$. We call the measure $\hm^X_L$ the elliptic measure with pole at $X$.

There is a Green function associated to $L$ in $\ree_+$, 
\[
G_L(X,Y): \ree_+ \times \ree_+ \setminus \{ X = Y \}   \to \re,
\]
which satisfies the following \cite[Lemma 2.6]{AGMT}. For fixed $X \in \ree_+$ the Green function can be extended, as a function in $Y$, to a function that vanishes continuously on the boundary $\rn$. The following `Riesz formula' holds and connects the elliptic measure and the Green function: If $f \in C_c^\infty(\rn)$ and $F \in C_c^\infty(\ree)$ are such that $F|_{\{0\} \times \rn}=f$, then
\begin{equation}\label{Rieszform.eq}
\int_{\rn} f(y) \, d\hm_L^X(y) - F(X) = - \iint_{\ree_+} A^T(Y)\nabla_{Y}G_L(X,Y) \cdot \nabla_Y F(Y) \, dY.
\end{equation}
\end{definition}

There is also a notion of Green function and elliptic measure with pole at infinity.
\begin{lemma}[{\cite[Lemma 3.5]{BTZ2}}]\label{def:Greeninf}
Let $L$ be a divergence form elliptic operator on $\ree_+$ with Green function $G_L(X,Y)$.
Let $Z_{k} = (2^{k},0)$ for $k$ a natural number, zero included. 
Define the sequence of functions $u_k: \overline{\ree_+} \to \re$ by 
\[u_k(Y) := \frac{G_L(Z_k, Y)}{G_L(Z_k, Z_0)},\]
where we have extended $u_k$ to the boundary by zero as $u_k(0,y) = 0$ for all $y \in \rn$. 
There exists a subsequence $u_{k_j}$ that converges uniformly on compact subsets of $\overline{\ree_+}$ 
to a function $U$ with the following properties.
\begin{itemize}
\item $U(0,y) = 0$ for all $y \in \rn$,
\item $U(1,0) = 1$,
\item $U(Y) > 0$ for all $Y \in \ree_+$,
\item $U \in C(\overline{\ree_+})$,
\item $U$ solves $L^T U = 0$ in $\ree_+$.
\end{itemize}
Moreover, there exists a locally finite doubling measure $\hm^\infty_L$ on $\rn$ with
\[\frac{1}{G(Z_{k_j},Z_0)} \hm^{Z_{k_j}}_L  \rightharpoonup   \hm^\infty_L \]
such that the following Riesz formula holds:
\begin{equation}\label{Rieszformulainfty.eq}
\int_{\rn} f(y) \, d\hm^\infty_L(y) = - \iint_{\ree_+} A^T(Y)\nabla_{Y}U(Y) \cdot \nabla_{Y} F(Y) \, dY,
\end{equation}
whenever $f \in C_c^\infty(\rn)$ and $F \in C_c^\infty(\ree)$ are such that $F|_{\{0\} \times \rn}=f$.
\end{lemma}

We call $\hm^\infty_L$ the {elliptic measure with pole at infinity} and $U$ the {Green function with pole at infinity}.

We define the following coefficients which measure the oscillation of an elliptic matrix on various regions. 

\begin{definition}[Oscillation coefficients]\label{OSCcoef.def}
Let $A$ be a $\Lambda$-elliptic matrix-valued function on $\ree_+$. 
For $x \in \rn$ and $r > 0$ we define
\[\alpha_2(r,x) := \inf_{A_0 \in \mathfrak{A}(\Lambda)}\left(\fiint_{(r/2,r] \times \Delta(x,r)} |A(s,y) - A_0|^2 \, dy ds \right)^{1/2},\]
where $\mathfrak{A}(\Lambda)$ denotes the collection of all {\it constant} $\Lambda$-elliptic matrices.
\end{definition}

Now we recall the weak DKP condition and the associated norm.

\begin{definition}[Weak DKP condition]
We say a $\Lambda$-elliptic matrix $A$ satisfies the {weak DKP condition} if $\nu$ defined by 
\[d\nu(r,x) := \alpha_2(r,x)^2 \, \frac{dx \, dr}{r}\]
is a Carleson measure. 
If $A$ satisfies the weak DKP condition, we call $\|\nu\|_{\mathcal{C}}$ the {weak DKP norm} of $A$.
\end{definition}
\begin{remark}
There are other similar oscillation coefficients that are either controlled by, or comparable to, $\alpha_2(r,x)$. For instance, the weak DKP condition here is implied by the original DKP condition. Since this is not the focus of the present work, we do not go into the details. See \cite{DPP, DLM, BTZ2}. 
\end{remark}


\begin{lemma}[{\cite[Theorem 4.1]{BTZ2}}]\label{BTZmain.thrm}
Let $A$ be a $\Lambda$-elliptic matrix-valued function on $\ree$ and 
let $\hm$ be the elliptic measure for $L = -\div A\nabla$ with pole at infinity.
Let $\varphi \ge 0$ be a non-zero smooth function with compact support in $\Delta(0,2)$ 
and define the measure 
\[d\mu(r,x) := |\nabla_x  \log (\varphi_r \ast \hm)(x)|^2 r \, dx \,dr = \frac{|((\nabla \varphi)_r \ast \hm)(x)|^2}{|(\varphi_r \ast \hm)(x)|^2} \frac{dx \, dr}{r}. \]
Suppose that $A$ satisfies the weak DKP condition and define 
\[d\nu(r,x) := \alpha_2(r,x)^2 \, \frac{dx \, dr}{r}\]
with $\alpha_2(r,x)$ the oscillation coefficients of $A$.
Then $\hm \ll \mathcal{L}^n$ and there exists a constant $C_3$ depending on $n$, $\Lambda$ and $\varphi$ such that
\[\|\mu\|_{\mathcal{C}} \le C_3 \|\nu\|_{\mathcal{C}}.\]
\end{lemma}
In the theorem $\mathcal{L}^n$ is $n$-dimensional Lebesgue measure 
and the statement $\hm \ll \mathcal{L}^n$ follows from the fact 
that $\|\mu\|_{\mathcal{C}} < \infty$, see Remark \ref{mtmeasures.rmk}. 

Although the heat kernel used in Section \ref{weights.sec}
does not satisfy the condition required from $\varphi$ in Theorem \ref{BTZmain.thrm},
the Carleson norms generated by different kernels are comparable.
We quote the following formulation of \cite[Corollary 12]{Koreycarl}.
\begin{lemma}\label{approxvsgaus.thrm}
Let $w$ be a doubling weight on $\mathbb{R}^n$,
let $\phi$ the heat kernel as in Section~\ref{weights.sec}
and let $\varphi$ be a non-trivial smooth function with support in $\Delta(0,2)$.
Define the measures $\mu_w$ and $\tilde{\mu}_w$ by
\[d\mu_w(r,x) := |\nabla_x \log(\phi_r \ast w)(x)|^2 r \, dx \, dr\]
and
\begin{equation}\label{mutilddef.eq}
d\tilde{\mu}_w(r,x) := |\nabla_x \log(\varphi_r \ast w)(x)|^2 r \, dx \, dr.
\end{equation}
Then $\|\mu_w\|_{\mathcal{C}} < \infty$ if and only if $\|\tilde{\mu}_w\|_{\mathcal{C}} < \infty$. Moreover, there exists a finite constant $C > 1$ depending on $n$, the doubling constant of $w$
and the Schwartz seminorms of $\varphi$ such that
\begin{equation}\label{mutmucomp.eq}
 \frac{1}{C} \|\mu_w\|_{\mathcal{C}} \le \|\tilde{\mu}_w\|_{\mathcal{C}}  \le C \|\mu_w\|_{\mathcal{C}}.
\end{equation} 
\end{lemma}

From Lemmata \ref{approxvsgaus.thrm} and \ref{BTZmain.thrm} and Theorem \ref{dkpsmallthrm.thrm}
we may immediately deduce the following, 
which is our main application of the present work.

\begin{theorem}\label{dkpsmallthrm.thrm}
Fix $\Lambda > 1$. 
There exist $\delta > 0$ and $C_3,C_4 > 0$ depending only on $n$, $\Lambda$
such that the following holds.
If $A$ is a $\Lambda$-elliptic matrix-valued function on $\ree_+$, 
$\hm$ is the elliptic measure for $L = -\div A\nabla$ with pole at infinity, 
and $A$ satisfies the weak DKP condition with 
\[\|\nu\|_{\mathcal{C}} \le \delta,\]
then $\hm \in A_\infty$ and 
\[
\log[\hm]_{A_\infty} \le C_3 \sqrt{\|\nu\|_{\mathcal{C}}}, \quad  \left\|\log \frac{d\hm}{dx} \right\|_{\BMO}   \le C_4 \sqrt[4]{\|\nu\|_{\mathcal{C}}},
\]
where $\nu$ is defined by
\[d\nu(r,x) := \alpha_2(r,x)^2 \, \frac{dx \, dr}{r},\]
$\alpha_2(r,x)$ being the oscillation coefficients of $A$. 
\end{theorem}

After the first version of this article had been posted, 
it was shown in \cite{DLM3} that one can push small constant $A_\infty$ results to the context of Lipschitz domains with small constant (as in \cite{DPP}) by using the `Dahlberg-Kenig-Stein pullback' \cite{Dahlpullback}
and further to even more exotic/rough geometric settings, as was done in the pioneering work of \cite{KT-Duke} for constant coefficient operators.  
The extension of these estimates to rough settings uses however a novel approach because the coefficients in \cite{DLM3} do not play nicely with subdomains.
\section{Further Remarks}\label{closermks.sect}
\subsection{Sharpness}
One may ask about the correct power dependence in Theorem~\ref{mainthrm.thrm}. 
For instance, whether it is true that 
\begin{equation}\label{mteqone.eq}
\log [w]_{A_{\infty}} \lesssim \no{\mu_w}_{\mathcal{C}}
\end{equation}
with constants only depending on dimension and doubling. In the dyadic case, one can track the constants in \cite{bucksum} to see that this is true. 
In fact, in the dyadic setting much more is known about these kinds of inequalities due to the Bellman function approach, see \cite{MR4411371} for overview and \cite{MR3293431} for a more specific example. The validity of \eqref{mteqone.eq} is an interesting question, and armed with \eqref{mteqone.eq}, one would obtain 
\begin{equation}\label{mteqtwo.eq}
\|\log w\|_{\BMO}   \lesssim \sqrt{\log [w]_{A_{\infty}}}  \lesssim \sqrt{\|\mu_w\|_{\mathcal{C}}},
\end{equation}
that is, a square root bound in Theorem \ref{mainthrm.thrm} of the form  $\no{\log w}_{BMO}  \lesssim \sqrt{\|\mu_w\|_{\mathcal{C}}}$. 
Indeed, it is shown in \cite{Kor2} that the first inequality in \eqref{mteqtwo.eq} is asymptotically sharp, see the remark on p.504 in \cite{Kor2}. 

\subsection{H\"older coefficients}
One may use the results here to investigate a variant of Theorem \ref{dkpsmallthrm.thrm} 
where the Carleson norm of $\nu$ vanishes at a particular rate rather than just qualitatively with respect to the scale.
A prime example of this is when the coefficient matrix $A$ is $\theta$-H\"older continuous with $0<\theta<1$.
Using either the hypothetical square root dependency instead of our actual weighted estimates,
one would be able to recover the correct $\theta$-H\"older estimates for the logarithm of the elliptic measure (by perhaps modifying the arguments in \cite{DLM} using this stronger $L^\infty$ control).
Alternatively,
the different $L^{1}$-type Carleson condition of Remark \ref{rmk:bad-carleson} can be used to the same effect.
Such results are not new as they follow from \cite{MR749677} or the more general results in \cite{MR3747493}\footnote{We thank Seick Kim for a helpful discussion regarding these results.},
and the techniques here combined with \cite{DLM} do not seem to provide considerable generalization.

\bibliography{DKPBESrefs}
\bibliographystyle{alpha}

\end{document}